\documentclass[11pt]{amsart}

\usepackage[margin=1.25in]{geometry}

\usepackage{graphicx}

\usepackage{wasysym}
\usepackage[misc]{ifsym}

\usepackage{amsmath, amssymb, amsthm, graphicx, enumerate, tikz, float, bbm, nicefrac}
\usepackage{mathtools, comment}
\usepackage{subcaption}
\usepackage[colorlinks]{hyperref}
\hypersetup{citecolor=blue}
\usetikzlibrary{matrix,arrows,decorations.pathmorphing}

\makeatletter
\@namedef{subjclassname@2020}{%
  \textup{2020} Mathematics Subject Classification}
\makeatother

\newtheorem{theorem}{Theorem}[section]
\newtheorem{lemma}[theorem]{Lemma}

\newtheorem{proposition}[theorem]{Proposition}

\theoremstyle{definition}
\newtheorem{definition}[theorem]{Definition}
\newtheorem{example}[theorem]{Example}

\newtheorem*{solution*}{Solution}

\theoremstyle{remark}
\newtheorem{remark}[theorem]{Remark}

\numberwithin{equation}{section}

\DeclareMathOperator{\Hom}{Hom}
\DeclareMathOperator{\Hg}{H}

\DeclareMathOperator{\id}{id}
\DeclareMathOperator{\HH}{HH}

\begin{document}

\allowdisplaybreaks

\title{Replacing bar-like resolutions in a simplicial setting}

\author[S. Carolus]{Samuel Carolus}
\address{Department of Defense}
\email{samcarolus@gmail.com}

\author[J. Laubacher]{Jacob Laubacher}
\address{Department of Mathematics, St. Norbert College, De Pere, Wisconsin 54115}
\email{jacob.laubacher@snc.edu}

\author[S.D. Vitalbo]{Sydney D. Vitalbo}
\address{Department of Mathematics, St. Norbert College, De Pere, Wisconsin 54115}
\email{sydney.vitalbo@snc.edu}

\author[L.K. Widlarz]{Leah K. Widlarz}
\address{Department of Mathematics, St. Norbert College, De Pere, Wisconsin 54115}
\email{leah.widlarz@snc.edu}


\date{\today}

\keywords{Simplicial modules, Hochschild homology, Resolutions.\\\indent\emph{Corresponding author.} Jacob Laubacher \Letter~\href{mailto:jacob.laubacher@snc.edu}{jacob.laubacher@snc.edu} \phone~920-403-2961.}

\thanks{The third and fourth authors received funding from the Poss-Wroble Fellowship at St. Norbert College}

\begin{abstract}
It is well known that the bar resolution can be replaced with any projective resolution of the corresponding algebra when computing the Hochschild (co)homology of that algebra. This is, in fact, a feature of its construction via derived functors. For generalizations and extensions of the Hochschild (co)homology, one uses a bar-like resolution in a simplicial setting in order to accommodate the changing module structures in every dimension. In this note, we present a method in order to replace these bar-like resolutions.
\end{abstract}

\maketitle

\section{Introduction}

Hochschild cohomology theory, first defined in \cite{H} to study extensions of algebras and then later employed to examine deformation theory in \cite{G}, was realized as a construction using derived functors by Cartan and Eilenberg in \cite{CE}. The fundamental ingredient in this construction used the so-called bar resolution associated to the given algebra, and then employing properties of derived functors, one can replace that bar resolution with any projective resolution. This is incredibly powerful for computation.

In 2016, Staic introduced a generalization of the Hochschild cohomology (see \cite{S}). The aptly-named secondary Hochschild cohomology was used to investigate deformations of the given algebra $A$ that had a nontrivial $B$-algebra structure. In \cite{LSS} it was discovered that the secondary Hochschild cohomology could be constructed using a bar-like resolution in a simplicial setting. Here the main hurdle to overcome was that the module structure changes in every dimension. Further generalizations, as well as the higher order Hochschild cohomology over the $d$-sphere, for example, have now also been constructed in this setting (see \cite{CL} and \cite{Laub}, among others).

Due to the changing module structure described above, the secondary Hochschild cohomology cannot be viewed as a derived functor. However, it would be advantageous if the property of being able to replace the bar-like resolution within the construction could be obtained. The goal of this paper is to answer that question. In Section \ref{prelims}, we recall all the necessary background information regarding the simplicial setting in which we work in order to keep this paper as self-contained as possible. Section \ref{main}, therefore, houses our main results. We introduce the proper morphisms and homotopies which help answer the question regarding the conditions required to replace these bar-like resolutions.

\section{Preliminaries}\label{prelims}

We fix $\mathbbm{k}$ to be a field, and we let $\otimes:=\otimes_\mathbbm{k}$. Next, we assume all $\mathbbm{k}$-algebras are associative and have multiplicative unit. Most of the results in this section are from \cite{LSS}.

\subsection{Simplicial algebras}

One can begin with a simplicial $\mathbbm{k}$-algebra, which is just a simplicial object in the category of $\mathbbm{k}$-algebras, and then build from there. The formal definition is below:

\begin{definition}(\cite{LSS})
A \textbf{simplicial $\mathbbm{k}$-algebra} $\mathcal{A}$ is a collection of $\mathbbm{k}$-algebras $\{\mathcal{A}_n\}_{n\geq0}$ together with morphisms of $\mathbbm{k}$-algebras $\delta_i^\mathcal{A}:\mathcal{A}_n\longrightarrow\mathcal{A}_{n-1}$ and $\sigma_i^\mathcal{A}:\mathcal{A}_n\longrightarrow\mathcal{A}_{n+1}$ for all $0\leq i\leq n$ such that
\begin{equation}\label{presimcon}
\delta_i^\mathcal{A}\delta_j^\mathcal{A}=\delta_{j-1}^\mathcal{A}\delta_i^\mathcal{A}
\end{equation}
whenever $i<j$, and
\begin{equation}\label{simcon}
\begin{gathered}
\sigma_i^\mathcal{A}\sigma_j^\mathcal{A}=\sigma_{j+1}^\mathcal{A}\sigma_i^\mathcal{A}\text{~~~~~if~~~~~}i\leq j,\\
\delta_i^\mathcal{A}\sigma_j^\mathcal{A}=\sigma_{j-1}^\mathcal{A}\delta_i^\mathcal{A}\text{~~~~~if~~~~~}i<j,\\
\delta_i^\mathcal{A}\sigma_j^\mathcal{A}=\id_{\mathcal{A}_n}\text{~~~~~if~~~~~}i=j\text{~~~~~or~~~~~}i=j+1,\text{~~~~~and~~~~~}\\
\delta_i^\mathcal{A}\sigma_j^\mathcal{A}=\sigma_j^\mathcal{A}\delta_{i-1}^\mathcal{A}\text{~~~~~if~~~~~}i>j+1.
\end{gathered}
\end{equation}
\end{definition}

It is sufficient to only have face maps and still be viewed as a chain complex. In fact, later on in this paper, it is much more convenient (and less cluttered) to view things from a presimplicial setting (satisfying \eqref{presimcon} only). However, there are times when it is crucial we also make use of the degeneracy maps, which require more conditions to be satisfied (see \eqref{simcon}).

\begin{example}(\cite{LSS})
Let $A$ be a $\mathbbm{k}$-algebra. Then $\mathcal{A}(A\otimes A^{op})$ is a simplicial $\mathbbm{k}$-algebra by setting $\mathcal{A}_n=A\otimes A^{op}$ for all $n\geq0$, with $\delta_i^\mathcal{A}=\id_{A\otimes A^{op}}$ and $\sigma_i^\mathcal{A}=\id_{A\otimes A^{op}}$ for all $0\leq i\leq n$.
\end{example}

\begin{example}(\cite{LSS})
Let $A$ be a $\mathbbm{k}$-algebra, $B$ a commutative $\mathbbm{k}$-algebra, and $\varepsilon:B\longrightarrow A$ a morphism of $\mathbbm{k}$-algebras such that $\varepsilon(B)\subseteq\mathcal{Z}(A)$. Then $\mathcal{A}(A,B,\varepsilon)$ is a simplicial $\mathbbm{k}$-algebra by setting $\mathcal{A}_n=A\otimes B^{\otimes2n+1}\otimes A^{op}$ for all $n\geq0$, with
\begin{align*}
\delta_0^\mathcal{A}&(a\otimes\alpha_1\otimes\cdots\otimes\alpha_n\otimes\gamma\otimes\beta_1\otimes\cdots\otimes\beta_n\otimes b)\\
&=a\varepsilon(\alpha_1)\otimes\alpha_2\otimes\cdots\otimes\alpha_n\otimes\gamma\beta_1\otimes\beta_2\otimes\cdots\otimes\beta_n\otimes b,
\end{align*}
\begin{align*}
\delta_i^\mathcal{A}&(a\otimes\alpha_1\otimes\cdots\otimes\alpha_n\otimes\gamma\otimes\beta_1\otimes\cdots\otimes\beta_n\otimes b)\\
&=a\otimes\alpha_1\otimes\cdots\otimes\alpha_i\alpha_{i+1}\otimes\cdots\otimes\alpha_n\otimes\gamma\otimes\beta_1\otimes\cdots\otimes\beta_i\beta_{i+1}\otimes\cdots\otimes\beta_n\otimes b
\end{align*}
for $1\leq i\leq n-1$, and
\begin{align*}
\delta_n^\mathcal{A}&(a\otimes\alpha_1\otimes\cdots\otimes\alpha_n\otimes\gamma\otimes\beta_1\otimes\cdots\otimes\beta_n\otimes b)\\
&=a\otimes\alpha_1\otimes\cdots\otimes\alpha_{n-1}\otimes\alpha_n\gamma\otimes\beta_1\otimes\cdots\otimes\beta_{n-1}\otimes\varepsilon(\beta_n)b,
\end{align*}
along with
\begin{align*}
\sigma_0^\mathcal{A}&(a\otimes\alpha_1\otimes\cdots\otimes\alpha_n\otimes\gamma\otimes\beta_1\otimes\cdots\otimes\beta_n\otimes b)\\
&=a\otimes1\otimes\alpha_1\otimes\cdots\otimes\alpha_n\otimes\gamma\otimes1\otimes\beta_1\otimes\cdots\otimes\beta_n\otimes b,
\end{align*}
\begin{align*}
\sigma_i^\mathcal{A}&(a\otimes\alpha_1\otimes\cdots\otimes\alpha_n\otimes\gamma\otimes\beta_1\otimes\cdots\otimes\beta_n\otimes b)\\
&=a\otimes\alpha_1\otimes\cdots\otimes\alpha_i\otimes1\otimes\alpha_{i+1}\otimes\cdots\otimes\alpha_n\otimes\gamma\otimes\beta_1\otimes\cdots\otimes\beta_i\otimes1\otimes\beta_{i+1}\otimes\cdots\otimes\beta_n\otimes b
\end{align*}
for $1\leq i\leq n-1$, and
\begin{align*}
\sigma_n^\mathcal{A}&(a\otimes\alpha_1\otimes\cdots\otimes\alpha_n\otimes\gamma\otimes\beta_1\otimes\cdots\otimes\beta_n\otimes b)\\
&=a\otimes\alpha_1\otimes\cdots\otimes\alpha_n\otimes1\otimes\gamma\otimes\beta_1\otimes\cdots\otimes\beta_n\otimes1\otimes b.
\end{align*}
\end{example}

For more examples, one can also define the simplicial $\mathbbm{k}$-algebras $\mathcal{A}^2(A)$ and $\mathcal{A}^3(A)$, which are used in the construction for the higher order Hochschild (co)homology over the $2$-sphere (\cite{Laub}) and the $3$-sphere (\cite{CL}), respectively, as well as $\mathcal{A}(\mathcal{Q})$, which is used to define the tertiary Hochschild (co)homology over a quintuple $\mathcal{Q}$ (studied in \cite{CHL} and \cite{CL}).

\subsection{Simplicial modules}

As the next natural step, we can then look at simplicial modules over these simplicial $\mathbbm{k}$-algebras.

\begin{definition}(\cite{LSS})
We say that $\mathcal{M}$ is a \textbf{simplicial left module} over the simplicial $\mathbbm{k}$-algebra $\mathcal{A}$ if $\mathcal{M}=\{\mathcal{M}_n\}_{n\geq0}$ is a simplicial $\mathbbm{k}$-vector space (satisfies \eqref{presimcon} and \eqref{simcon}) together with a left $\mathcal{A}_n$-module structure on $\mathcal{M}_n$ for all $n\geq0$ such that we have the following natural compatibility conditions:
\begin{equation*}
\delta_i^\mathcal{M}(a_nm_n)=\delta_i^\mathcal{A}(a_n)\delta_i^\mathcal{M}(m_n)
\end{equation*}
and
\begin{equation*}
\sigma_i^\mathcal{M}(a_nm_n)=\sigma_i^\mathcal{A}(a_n)\sigma_i^\mathcal{M}(m_n)
\end{equation*}
for all $a_n\in\mathcal{A}_n$, for all $m_n\in\mathcal{M}_n$, and for all $0\leq i\leq n$.
\end{definition}

One can easily define a simplicial right module and a cosimplicial left module over a simplicial $\mathbbm{k}$-algebra in an analogous way.

\begin{example}[The bar resolution](\cite{LSS})\label{exbar}
Let $A$ be a $\mathbbm{k}$-algebra. Then $\mathcal{B}(A)$ is a simplicial left module over the simplicial $\mathbbm{k}$-algebra $\mathcal{A}(A\otimes A^{op})$ by setting $\mathcal{B}_n=A^{\otimes n+2}$ for all $n\geq0$ with the left $\mathcal{A}_n$-module structure given by
$$
(a\otimes b)\cdot(a_0\otimes a_1\otimes\cdots\otimes a_n\otimes a_{n+1})=aa_0\otimes a_1\otimes\cdots\otimes a_n\otimes a_{n+1}b.
$$
Moreover, we have that
$$
\delta_i^\mathcal{B}(a_0\otimes a_1\otimes\cdots\otimes a_n\otimes a_{n+1})=a_0\otimes\cdots\otimes a_ia_{i+1}\otimes\cdots\otimes a_{n+1}
$$
and
$$
\sigma_i^\mathcal{B}(a_0\otimes a_1\otimes\cdots\otimes a_n\otimes a_{n+1})=a_0\otimes\cdots\otimes a_i\otimes1\otimes a_{i+1}\otimes\cdots\otimes a_{n+1}
$$
for all $0\leq i\leq n$.
\end{example}

\begin{example}(\cite{LSS})
Let $A$ be a $\mathbbm{k}$-algebra and $M$ an $A$-bimodule. Then $\mathcal{M}(M)$ is a simplicial right module over the simplicial $\mathbbm{k}$-algebra $\mathcal{A}(A\otimes A^{op})$ by setting $\mathcal{M}_n=M$ for all $n\geq0$ with the right $\mathcal{A}_n$-module structure given by $m\cdot(a\otimes b)=bma$. Moreover, we have that $\delta_i^\mathcal{M}=\id_M$ and $\sigma_i^\mathcal{M}=\id_M$ for all $0\leq i\leq n$.
\end{example}

\begin{example}[The secondary bar resolution](\cite{LSS})\label{exsecbar}
Let $A$ be a $\mathbbm{k}$-algebra, $B$ a commutative $\mathbbm{k}$-algebra, and $\varepsilon:B\longrightarrow A$ a morphism of $\mathbbm{k}$-algebras such that $\varepsilon(B)\subseteq\mathcal{Z}(A)$. Then $\mathcal{B}(A,B,\varepsilon)$ is a simplicial left module over the simplicial $\mathbbm{k}$-algebra $\mathcal{A}(A,B,\varepsilon)$ by setting $\mathcal{B}_n=A^{\otimes n+2}\otimes B^{\otimes\frac{(n+1)(n+2)}{2}}$ for all $n\geq0$ with the left $\mathcal{A}_n$-module structure given by
$$
(a\otimes \alpha_1\otimes\cdots\otimes \alpha_n\otimes\gamma\otimes\beta_1\otimes\cdots\otimes\beta_n\otimes b)\cdot\left(\otimes
\begin{pmatrix}
a_0 & b_{0,1} & \cdots & b_{0,n} & b_{0,n+1}\\
1 & a_1 & \cdots & b_{1,n} & b_{1,n+1}\\
\vdots & \vdots & \ddots & \vdots & \vdots\\
1 & 1 & \cdots & a_n & b_{n,n+1}\\
1 & 1 & \cdots & 1 & a_{n+1}\\
\end{pmatrix}
\right)
$$
$$
=\otimes
\begin{pmatrix}
aa_0 & \alpha_1b_{0,1} & \cdots & \alpha_nb_{0,n} & \gamma b_{0,n+1}\\
1 & a_1 & \cdots & b_{1,n} & b_{1,n+1}\beta_1\\
\vdots & \vdots & \ddots & \vdots & \vdots\\
1 & 1 & \cdots & a_n & b_{n,n+1}\beta_n\\
1 & 1 & \cdots & 1 & a_{n+1}b\\
\end{pmatrix}.
$$
Moreover, we have that
$$
\delta_i^{\mathcal{B}}
\left(\otimes
\begin{pmatrix}
a_0 & b_{0,1} & \cdots & b_{0,n} & b_{0,n+1}\\
1 & a_1 & \cdots & b_{1,n} & b_{1,n+1}\\
\vdots & \vdots & \ddots & \vdots & \vdots\\
1 & 1 & \cdots & a_n & b_{n,n+1}\\
1 & 1 & \cdots & 1 & a_{n+1}\\
\end{pmatrix}
\right)
$$
$$
=
\otimes
\begin{pmatrix}
a_0 & b_{0,1} & \cdots & b_{0,i}b_{0,i+1} & \cdots & b_{0,n} & b_{0,n+1}\\
1 & a_1 & \cdots & b_{1,i}b_{1,i+1} & \cdots & b_{1,n} & b_{1,n+1}\\
\vdots & \vdots & \ddots & \vdots & \ddots & \vdots & \vdots\\
1 & 1 & \cdots & a_i\varepsilon(b_{i,i+1})a_{i+1} & \cdots & b_{i,n}b_{i+1,n} & b_{i,n+1}b_{i+1,n+1}\\
\vdots & \vdots & \ddots & \vdots & \ddots & \vdots & \vdots\\
1 & 1 & \cdots & 1 & \cdots & a_n & b_{n,n+1}\\
1 & 1 & \cdots & 1 & \cdots & 1 & a_{n+1}\\
\end{pmatrix}
$$
and
$$
\sigma_i^{\mathcal{B}}
\left(\otimes
\begin{pmatrix}
a_0 & b_{0,1} & \cdots & b_{0,n} & b_{0,n+1}\\
1 & a_1 & \cdots & b_{1,n} & b_{1,n+1}\\
\vdots & \vdots & \ddots & \vdots & \vdots\\
1 & 1 & \cdots & a_n & b_{n,n+1}\\
1 & 1 & \cdots & 1 & a_{n+1}\\
\end{pmatrix}
\right)
$$
$$
=
\otimes
\begin{pmatrix}
a_0 & b_{0,1} & \cdots & b_{0,i} & 1 & b_{0,i+1} & \cdots & b_{0,n} & b_{0,n+1}\\
1 & a_1 & \cdots & b_{1,i} & 1 & b_{1,i+1} & \cdots & b_{1,n} & b_{1,n+1}\\
\vdots & \vdots & \ddots & \vdots & \vdots & \vdots & \ddots & \vdots & \vdots\\
1 & 1 & \cdots & a_i & 1 & b_{i,i+1} & \cdots & b_{i,n} & b_{i,n+1}\\
1 & 1 & \cdots & 1 & 1 & 1 & \cdots & 1 & 1\\
1 & 1 & \cdots & 1 & 1 & a_{i+1} & \cdots & b_{i+1,n} & b_{i+1,n+1}\\
\vdots & \vdots & \ddots & \vdots & \vdots & \vdots & \ddots & \vdots & \vdots\\
1 & 1 & \cdots & 1 & 1 & 1 & \cdots & a_n & b_{n,n+1}\\
1 & 1 & \cdots & 1 & 1 & 1 & \cdots & 1 & a_{n+1}\\
\end{pmatrix}
$$
for all $0\leq i\leq n$.
\end{example}

We dwell on Example \ref{exsecbar} specifically, and how it contrasts to Example \ref{exbar}. Notice how the secondary bar resolution has a different module structure in each dimension, whereas the classic bar resolution has a consistent module structure throughout. However, both can be represented in this context as simplicial modules over simplicial algebras.

\begin{example}(\cite{LSS})
Let $A$ be a $\mathbbm{k}$-algebra, $B$ a commutative $\mathbbm{k}$-algebra, and $\varepsilon:B\longrightarrow A$ a morphism of $\mathbbm{k}$-algebras such that $\varepsilon(B)\subseteq\mathcal{Z}(A)$. Furthermore, let $M$ be an $A$-bimodule which is $B$-symmetric. Then $\mathcal{S}(M)$ is a simplicial right module over the simplicial $\mathbbm{k}$-algebra $\mathcal{A}(A,B,\varepsilon)$ by setting $\mathcal{S}_n=M$ for all $n\geq0$ with the right $\mathcal{A}_n$-module structure given by
$$
m\cdot(a\otimes \alpha_1\otimes\cdots\otimes \alpha_n\otimes\gamma\otimes\beta_1\otimes\cdots\otimes\beta_n\otimes b)=bma\varepsilon(\alpha_1\cdots\alpha_n\gamma\beta_1\cdots\beta_n).
$$
Moreover, we have that $\delta_i^\mathcal{S}=\id_M$ and $\sigma_i^\mathcal{S}=\id_M$ for all $0\leq i\leq n$.
\end{example}

One can also define other examples like $\mathcal{M}^2(M)$ and $\mathcal{B}^2(A)$ (see \cite{Laub}), $\mathcal{M}^3(M)$, $\mathcal{B}^3(A)$, $\mathcal{S}(\mathcal{Q})$, and $\mathcal{B}(\mathcal{Q})$ (see \cite{CL}), and $\mathcal{L}(A,B,\varepsilon)$ (see \cite{LSS}). These also have varying module structures in every dimension.

\subsection{Combining simplicial modules}

Finally, we recall the aptly named Tensor Lemma, which takes the role of the Tor functor. As is necessary, this allows us to accommodate a changing module structure in each dimension.

\begin{lemma}[Tensor Lemma]\emph{(\cite{LSS})}\label{TL}
Let $(\mathcal{X}, \delta_i^\mathcal{X}, \sigma_i^\mathcal{X})$ be a simplicial right module, and let $(\mathcal{Y}, \delta_i^\mathcal{Y}, \sigma_i^\mathcal{Y})$ be a simplicial left module, both over the simplicial $\mathbbm{k}$-algebra $\mathcal{A}$. Then $\mathcal{M}=(\mathcal{X}\otimes_\mathcal{A}\mathcal{Y}, D_i, S_i)$ is a simplicial $\mathbbm{k}$-module (satisfies \eqref{presimcon} and \eqref{simcon}) where $\mathcal{M}_n=X_n\otimes_{A_n}Y_n$ for all $n\geq0$, and we take
$$
D_i:\mathcal{M}_n\longrightarrow\mathcal{M}_{n-1}
$$
given by
$$
D_i(x_n\otimes_{A_n}y_n)=\delta_i^\mathcal{X}(x_n)\otimes_{A_{n-1}}\delta_i^\mathcal{Y}(y_n),
$$
and
$$
S_i:\mathcal{M}_n\longrightarrow\mathcal{M}_{n+1}
$$
given by
$$
S_i(x_n\otimes_{A_n}y_n)=\sigma_i^\mathcal{X}(x_n)\otimes_{A_{n+1}}\sigma_i^\mathcal{Y}(y_n),
$$
for all $0\leq i\leq n$.
\end{lemma}

The Hom Lemma from \cite{LSS}, which we omit here, produces a cochain complex in the same context. It combines a simplicial left module $\mathcal{X}$ and a cosimplicial left module $\mathcal{Y}$ over a common simplicial $\mathbbm{k}$-algebra $\mathcal{A}$, generating a cosimplicial $\mathbbm{k}$-module denoted by $\Hom_\mathcal{A}(\mathcal{X},\mathcal{Y})$.

\begin{example}(\cite{LSS})\label{ex1}
Using these simplicial structures and the Tensor Lemma \ref{TL}, one can see that the Hochschild homology of $A$ with coefficients in $M$, denoted $\Hg_\bullet(A,M)$, can be viewed as the homology of the chain complex $\mathcal{M}(M)\otimes_{\mathcal{A}(A\otimes A^{op})}\mathcal{B}(A)$. In notation, we have that
$$
\Hg_\bullet(A,M)=\Hg_\bullet(\mathcal{M}(M)\otimes_{\mathcal{A}(A\otimes A^{op})}\mathcal{B}(A)).
$$
\end{example}

\begin{example}(\cite{LSS})\label{ex2}
Again using the Tensor Lemma \ref{TL}, along with previously established examples, one gets that the secondary Hochschild homology of the triple $(A,B,\varepsilon)$ with coefficients in $M$ (introduced in \cite{S} and more thoroughly investigated in \cite{L}, and its cohomology in \cite{SS}), denoted $\Hg_\bullet((A,B,\varepsilon);M)$, can be realized as the homology of the chain complex $\mathcal{S}(M)\otimes_{\mathcal{A}(A,B,\varepsilon)}\mathcal{B}(A,B,\varepsilon)$. Notationally, we have that
$$
\Hg_\bullet((A,B,\varepsilon);M)=\Hg_\bullet(\mathcal{S}(M)\otimes_{\mathcal{A}(A,B,\varepsilon)}\mathcal{B}(A,B,\varepsilon)).
$$
\end{example}

One can also see that the secondary Hochschild homology associated to a triple $(A,B,\varepsilon)$ is $\HH_\bullet(A,B,\varepsilon)=\Hg_\bullet(\mathcal{L}(A,B,\varepsilon)\otimes_{\mathcal{A}(A,B,\varepsilon)}\mathcal{B}(A,B,\varepsilon))$ (defined in \cite{LSS}, and further studied in \cite{L2}, and its cohomology in \cite{BHL}), that the higher order Hochschild homology over the $d$-sphere is $\Hg_\bullet^{S^d}(A,M)=\Hg_\bullet(\mathcal{M}^d(M)\otimes_{\mathcal{A}^d(A)}\mathcal{B}^d(A))$ for $d=2$ (\cite{Laub}) and $d=3$ (\cite{CL}), and that the tertiary Hochschild homology of the quintuple $\mathcal{Q}$ with coefficients in $M$ is $\Hg_\bullet(\mathcal{Q};M)=\Hg_\bullet(\mathcal{S}(\mathcal{Q})\otimes_{\mathcal{A}(\mathcal{Q})}\mathcal{B}(\mathcal{Q}))$ (see \cite{CL}).

\section{Main Results}\label{main}

For the sake of computation, it is natural to wonder when we can replace a simplicial left module. In particular, following the many examples above, our goal is to replace the bar-like resolutions (like $\mathcal{B}(A)$, $\mathcal{B}(A,B,\varepsilon)$, $\mathcal{B}^2(A)$, $\mathcal{B}^3(A)$, and $\mathcal{B}(\mathcal{Q})$). The aim of this section is to approach that topic.

We fix $\mathcal{A}$ to be a simplicial $\mathbbm{k}$-algebra and $\mathcal{B}$ and $\mathcal{C}$ to be simplicial left modules over $\mathcal{A}$.

\begin{definition}
We say that $f:\mathcal{B}\longrightarrow\mathcal{C}$ is a \textbf{presimplicial morphism} if there is a family of $\mathcal{A}_n$-module morphisms $f_n:\mathcal{B}_n\longrightarrow\mathcal{C}_n$ such that
\begin{equation}\label{SM}
f_{n-1}\delta_i^\mathcal{B}=\delta_i^\mathcal{C} f_n
\end{equation}
for all $n\geq0$ and $0\leq i\leq n$.
\end{definition}

\begin{definition}\label{PreSH}
A \textbf{presimplicial homotopy $h$} between two presimplicial morphisms $f,g:\mathcal{B}\longrightarrow\mathcal{C}$ is a family of maps $h_i:\mathcal{B}_n\longrightarrow\mathcal{C}_{n+1}$ for $0\leq i\leq n$ and $n\geq0$ such that
\begin{enumerate}[(i)]
\item $h_i(a_nb_n)=\sigma_i^{\mathcal{A}}(a_n)h_i(b_n)$ for all $a_n\in\mathcal{A}_n$ and $b_n\in\mathcal{B}_n$, and
\item the following are satisfied:
\begin{equation}\label{SH}
\begin{gathered}
\delta_i^{\mathcal{C}}h_j=h_{j-1}\delta_i^{\mathcal{B}}\hspace{.15in}\text{for}\hspace{.15in}i<j,\\
\delta_i^{\mathcal{C}}h_i=\delta_i^{\mathcal{C}}h_{i-1}\hspace{.15in}\text{for}\hspace{.15in}0<i\leq n,\\
\delta_i^{\mathcal{C}}h_j=h_j\delta_{i-1}^{\mathcal{B}}\hspace{.15in}\text{for}\hspace{.15in}i>j+1,\\
\delta_0^{\mathcal{C}}h_0=f_n,\\
\delta_{n+1}^{\mathcal{C}}h_n=g_n.
\end{gathered}
\end{equation}
\end{enumerate}
\end{definition}

\begin{remark}
In the context of Definition \ref{PreSH}, it appropriate to say that the two presimplicial morphisms $f:\mathcal{B}\longrightarrow\mathcal{C}$ and $g:\mathcal{B}\longrightarrow\mathcal{C}$ are \emph{presimplicially homotopic}. When this is the case, we will write $f\sim g$.
\end{remark}

\begin{lemma}
Presimplicial homotopy defines an equivalence relation.
\end{lemma}
\begin{proof}
First observe that since $\mathcal{C}$ is a simplicial left module over $\mathcal{A}$, it is endowed with degeneracy maps $\sigma_i^{\mathcal{C}}$. These are necessary in what follows.

To see that $f\sim f$, consider $h$ with the family of maps $h_i:B_n\longrightarrow C_{n+1}$ for $0\leq i\leq n$ given by $h_i:=\sigma_i^{\mathcal{C}}f_n$. One can then check the necessary conditions in Definition \ref{PreSH} are satisfied in order to obtain reflexivity.

For symmetry, we first suppose that $f\sim g$ under the presimplicial homotopy $h$. To see that $g\sim f$ we take $t$ as follows: define the family of maps $t_i:B_n\longrightarrow C_{n+1}$ for $0\leq i\leq n$ by $t_i:=\sigma_i^{\mathcal{C}}(f_n+g_n)-h_i$.

Finally, for transitivity, we suppose that $f\sim g$ under the presimplicial homotopy $h$, and we suppose that $g\sim l$ under the presimplicial homotopy $t$. To see that $f\sim l$, we define $s$ with the family of maps $s_i:B_n\longrightarrow C_{n+1}$ for $0\leq i\leq n$ by $s_i:=h_i+t_i-\sigma_i^{\mathcal{C}}g_n$.

Hence, presimplicial homotopy defines an equivalence relation.
\end{proof}

\begin{definition}\label{PSHE}
We say that a presimplicial morphism $f:\mathcal{B}\longrightarrow\mathcal{C}$ is a \textbf{presimplicial homotopy equivalence} if there exists a presimplicial morphism $g:\mathcal{C}\longrightarrow\mathcal{B}$ such that there is a presimplicial homotopy between $gf$ and $\id_{\mathcal{B}}$, and a presimplicial homotopy between $fg$ and $\id_{\mathcal{C}}$.
\end{definition}

\subsection{Replacing the simplicial left modules}

Here we present a notion of how to replace these bar-like resolutions.

\begin{proposition}\label{Replace}
Let $\mathcal{M}$ be a simplicial right module over $\mathcal{A}$. If there exists a presimplicial homotopy equivalence $f:\mathcal{B}\longrightarrow\mathcal{C}$, then $$\Hg_\bullet(\mathcal{M}\otimes_\mathcal{A}\mathcal{B})\cong\Hg_\bullet(\mathcal{M}\otimes_\mathcal{A}\mathcal{C}).$$
\end{proposition}
\begin{proof}
We are given presimplicial morphisms $f:\mathcal{B}\longrightarrow\mathcal{C}$ and $g:\mathcal{C}\longrightarrow\mathcal{B}$ such that $gf$ is presimplicially homotopic to $\id_{\mathcal{B}}$ and $fg$ is presimplicially homotopic to $\id_{\mathcal{C}}$. Thus, for all $n\geq0$ and $0\leq i\leq n$ there exists maps $h_i:B_n\longrightarrow B_{n+1}$ and $t_i:C_n\longrightarrow C_{n+1}$ such that each satisfies Definition \ref{PreSH} appropriately. Without loss of generality, we say that $\delta_0^{\mathcal{B}}h_0(b)=g_nf_n(b)$ and $\delta_{n+1}^{\mathcal{B}}h_n(b)=b$. Likewise $\delta_0^{\mathcal{C}}t_0(c)=f_ng_n(c)$ and $\delta_{n+1}^{\mathcal{C}}t_n(c)=c$.

Recall that $D_i^{\mathcal{B}}:M_n\otimes_{A_n}B_n\longrightarrow M_{n-1}\otimes_{A_{n-1}}B_{n-1}$ is given by $$D_i^{\mathcal{B}}(m\otimes_{A_n}b)=\delta_i^{\mathcal{M}}(m)\otimes_{A_{n-1}}\delta_i^{\mathcal{B}}(b)$$ for all $n\geq0$ and $0\leq i\leq n$. Likewise for $D_i^{\mathcal{C}}:M_n\otimes_{A_n}C_n\longrightarrow M_{n-1}\otimes_{A_{n-1}}C_{n-1}$.

Define $h_i':M_n\otimes_{A_n}B_n\longrightarrow M_{n+1}\otimes_{A_{n+1}}B_{n+1}$ by $$h_i'(m\otimes_{A_n}b)=\sigma_i^{\mathcal{M}}(m)\otimes_{A_{n+1}}h_i(b)$$ for $0\leq i\leq n$. One can check that $h_i'$ is well-defined. Likewise for $t_i':M_n\otimes_{A_n}C_n\longrightarrow M_{n+1}\otimes_{A_{n+1}}C_{n+1}$.

Next, define $F:\mathcal{M}\otimes_{\mathcal{A}}\mathcal{B}\longrightarrow\mathcal{M}\otimes_{\mathcal{A}}\mathcal{C}$ by $$F_n(m\otimes_{A_n}b)=m\otimes_{A_n}f_n(b).$$ Notice that for all $n\geq0$ and $0\leq i\leq n$ we have that $$F_{n-1}D_i^{\mathcal{B}}(m\otimes_{A_n}b)=F_{n-1}(\delta_i^{\mathcal{M}}(m)\otimes_{A_{n-1}}\delta_i^{\mathcal{B}}(b))=\delta_i^{\mathcal{M}}(m)\otimes_{A_{n-1}}f_{n-1}\delta_i^{\mathcal{B}}(b),$$ and $$D_i^{\mathcal{C}}F_n(m\otimes_{A_n}b)=D_i^{\mathcal{C}}(m\otimes_{A_n}f_n(b))=\delta_i^{\mathcal{M}}(m)\otimes_{A_{n-1}}\delta_i^{\mathcal{C}}f_n(b),$$ which are equal due to \eqref{SM}. One can verify that $F$ is well-defined. Hence $F$ is a morphism of presimplicial $\mathbbm{k}$-modules. Likewise for $G:\mathcal{M}\otimes_{\mathcal{A}}\mathcal{C}\longrightarrow\mathcal{M}\otimes_{\mathcal{A}}\mathcal{B}$.

Our goal will be to show that $h_i'$ is a presimplicial homotopy (as $\mathbbm{k}$-modules) between $GF$ and $\id_{\mathcal{M}\otimes_{\mathcal{A}}\mathcal{B}}$ (from $\mathcal{M}\otimes_{\mathcal{A}}\mathcal{B}$ to $\mathcal{M}\otimes_{\mathcal{A}}\mathcal{B}$). Likewise for $t_i'$ between $FG$ and $\id_{\mathcal{M}\otimes_{\mathcal{A}}\mathcal{C}}$ (from $\mathcal{M}\otimes_{\mathcal{A}}\mathcal{C}$ to $\mathcal{M}\otimes_{\mathcal{A}}\mathcal{C}$).

Observe that for all $n\geq0$ we have the conditions in \eqref{SH} satisfied as $\mathbbm{k}$-modules. That is,
\begin{align*}
D_i^{\mathcal{B}}h_j'(m\otimes_{A_n}b)&=D_i^{\mathcal{B}}(\sigma_j^{\mathcal{M}}(m)\otimes_{A_{n+1}}h_j(b))\\
&=\delta_i^{\mathcal{M}}\sigma_j^{\mathcal{M}}(m)\otimes_{A_n}\delta_i^{\mathcal{B}}h_j(b)\\
&=\sigma_{j-1}^{\mathcal{M}}\delta_i^{\mathcal{M}}(m)\otimes_{A_n}h_{j-1}\delta_i^{\mathcal{B}}(b)\\
&=h_{j-1}'(\delta_i^{\mathcal{M}}(m)\otimes_{A_{n-1}}\delta_i^{\mathcal{B}}(b))\\
&=h_{j-1}'D_i^{\mathcal{B}}(m\otimes_{A_n}b)
\end{align*}
for $i<j$. Next, for $0<i\leq n$ we have
\begin{align*}
D_i^{\mathcal{B}}h_i'(m\otimes_{A_n}b)&=D_i^{\mathcal{B}}(\sigma_i^{\mathcal{M}}(m)\otimes_{A_{n+1}}h_i(b))\\
&=\delta_i^{\mathcal{M}}\sigma_i^{\mathcal{M}}(m)\otimes_{A_n}\delta_i^{\mathcal{B}}h_i(b)\\
&=\delta_i^{\mathcal{M}}\sigma_{i-1}^{\mathcal{M}}(m)\otimes_{A_n}\delta_i^{\mathcal{B}}h_{i-1}(b)\\
&=D_i^{\mathcal{B}}(\sigma_{i-1}^{\mathcal{M}}(m)\otimes_{A_{n+1}}h_{i-1}(b))\\
&=D_i^{\mathcal{B}}h_{i-1}'(m\otimes_{A_n}b).
\end{align*}
Furthermore, we have that
\begin{align*}
D_i^{\mathcal{B}}h_j'(m\otimes_{A_n}b)&=D_i^{\mathcal{B}}(\sigma_j^{\mathcal{M}}(m)\otimes_{A_{n+1}}h_j(b))\\
&=\delta_i^{\mathcal{M}}\sigma_j^{\mathcal{M}}(m)\otimes_{A_n}\delta_i^{\mathcal{B}}h_j(b)\\
&=\sigma_j^{\mathcal{M}}\delta_{i-1}^{\mathcal{M}}(m)\otimes_{A_n}h_j\delta_{i-1}^{\mathcal{B}}(b)\\
&=h_j'(\delta_{i-1}^{\mathcal{M}}(m)\otimes_{A_{n-1}}\delta_{i-1}^{\mathcal{B}}(b))\\
&=h_j'D_{i-1}^{\mathcal{B}}(m\otimes_{A_n}b)
\end{align*}
whenever $i>j+1$. Finally we have that
\begin{align*}
D_0^{\mathcal{B}}h_0'(m\otimes_{A_n}b)&=D_0^{\mathcal{B}}(\sigma_0^{\mathcal{M}}(m)\otimes_{A_{n+1}}h_0(b))\\
&=\delta_0^{\mathcal{M}}\sigma_0^{\mathcal{M}}(m)\otimes_{A_n}\delta_0^{\mathcal{B}}h_0(b)\\
&=m\otimes_{A_n}g_nf_n(b)\\
&=G_n(m\otimes_{A_n}f_n(b))\\
&=G_nF_n(m\otimes_{A_n}b)
\end{align*}
and
\begin{align*}
D_{n+1}^{\mathcal{B}}h_n'(m\otimes_{A_n}b)&=D_{n+1}^{\mathcal{B}}(\sigma_n^{\mathcal{M}}(m)\otimes_{A_{n+1}}h_n(b))\\
&=\delta_{n+1}^{\mathcal{M}}\sigma_n^{\mathcal{M}}(m)\otimes_{A_n}\delta_{n+1}^{\mathcal{B}}h_n(b)\\
&=m\otimes_{A_n}b\\
&=\id_{\mathcal{M}\otimes_{\mathcal{A}}\mathcal{B}}(m\otimes_{A_n}b).
\end{align*}

Thus $h_i'$ is a presimplicial homotopy between $GF$ and $\id_{\mathcal{M}\otimes_{\mathcal{A}}\mathcal{B}}$ (as $\mathbbm{k}$-modules). Similarly we get that $t_i'$ is a presimplicial homotopy between $FG$ and $\id_{\mathcal{M}\otimes_{\mathcal{A}}\mathcal{C}}$ (as $\mathbbm{k}$-modules). Hence, the result follows.
\end{proof}

Following Proposition \ref{Replace}, one can then compute the homology of chain complexes built in this context with an appropriate replacement. Specifically, referencing Example \ref{ex1} and the usual Hochschild homology, if there was a simplicial module $\mathcal{C}$ such that there exists a presimplicial homotopy equivalence $f:\mathcal{B}(A)\longrightarrow\mathcal{C}$, then we have that
$$
\Hg_\bullet(A,M)=\Hg_\bullet(\mathcal{M}(M)\otimes_{\mathcal{A}(A\otimes A^{op})}\mathcal{B}(A))\cong\Hg_\bullet(\mathcal{M}(M)\otimes_{\mathcal{A}(A\otimes A^{op})}\mathcal{C}).
$$
Likewise for Example \ref{ex2}, if there was a simplicial module $\mathcal{D}$ such that there exists a presimplicial homotopy equivalence $f:\mathcal{B}(A,B,\varepsilon)\longrightarrow\mathcal{D}$, then one could compute the secondary Hochschild homology in an alternative way. In notation, we would have
$$
\Hg_\bullet((A,B,\varepsilon);M)=\Hg_\bullet(\mathcal{S}(M)\otimes_{\mathcal{A}(A,B,\varepsilon)}\mathcal{B}(A,B,\varepsilon))\cong\Hg_\bullet(\mathcal{S}(M)\otimes_{\mathcal{A}(A,B,\varepsilon)}\mathcal{D}).
$$

One can get similar results using replacements via presimplicial homotopy equivalence for the rest of the examples discussed at the end of Section \ref{prelims}. Finally, for the sake of completion, switching to the context of cosimplicial modules by way of the Hom Lemma (discussed in \cite{LSS}), we get an analogous result.

\begin{proposition}
Let $\mathcal{M}$ be a cosimplicial left module over $\mathcal{A}$. If there exists a presimplicial homotopy equivalence $f:\mathcal{B}\longrightarrow\mathcal{C}$, then $$\Hg^\bullet(\Hom_{\mathcal{A}}(\mathcal{B},\mathcal{M}))\cong\Hg^\bullet(\Hom_{\mathcal{A}}(\mathcal{C},\mathcal{M})).$$
\end{proposition}
\begin{proof}
Similar to that of Proposition \ref{Replace}.
\end{proof}


\end{document}